\documentclass[12pt]{article}
\usepackage{graphicx} 
\usepackage{subcaption}
\usepackage{float} 
\usepackage{amssymb}
\usepackage{amsmath}
\usepackage{enumitem} 
\usepackage{mathtools}
\usepackage[thmmarks,amsmath]{ntheorem} 
\usepackage{lineno}
\usepackage{xcolor}
\definecolor{darkblue}{RGB}{0, 0, 100}
\definecolor{darkgreen}{RGB}{0, 100, 0}

\newcommand{\mad}{\operatorname{mad}}
\newcommand{\Prob}{\mathbb P}
\newcommand{\Ex}{\mathbb E}

\usepackage[numbers]{natbib}  
\usepackage{bm} 
\usepackage{extarrows}
\usepackage{mathrsfs} 

\usepackage{lipsum}


\usepackage{booktabs} 
{
    \theoremstyle{nonumberplain}
    \theoremheaderfont{\bfseries}
    \theorembodyfont{\normalfont}
    \theoremsymbol{\mbox{$\Box$}}
    \newtheorem{proof}{Proof}
}
\makeatletter
\newcounter{mycounter}[section]

{
    \theoremstyle{plain}
    \theoremheaderfont{\bfseries}
    \theorembodyfont{\normalfont}
    \newtheorem{theorem}[mycounter]{Theorem}
    
    \newtheorem{lemma}[mycounter]{Lemma}

    \newtheorem{corollary}[mycounter]{Corollary}
    
    \newtheorem{claim}[mycounter]{Claim}

    \newtheorem{problem}[mycounter]{Problem}
    \newtheorem{conjecture}[mycounter]{Conjecture}
}

{
    \theoremheaderfont{\bfseries}
    \theorembodyfont{\normalfont}
    
}

\makeatother

\newcommand{\RNum}[1]{\uppercase\expandafter{\romannumeral #1\relax}}
\allowdisplaybreaks[4]

\graphicspath{{figure/}}                                 
\usepackage[a4paper]{geometry}                        
\begin{document}
\title{A linear bound for a connectivity partition in graphs
\thanks{The author's work is supported by National Natural Science Foundation of China (No.12571373), NSFC (No.12461064)}}

\author{Jia Zhou\textsuperscript{1},  Jin Yan\textsuperscript{2}, Yunshu Gao\textsuperscript{1}\footnote{Corresponding author. E-mail address: gysh2004@gmail.com.} %
\unskip\\[2mm]
\textsuperscript{1} School of Mathematics and Statistics, Ningxia University,\\ Yinchuan 750021, China\\
\textsuperscript{2} School of Mathematics, Shandong University, Jinan 250100, China}

\date{}
\maketitle

\begin{abstract}
K\"uhn and Osthus proved that, for every positive integer $\ell$, every $2^{16}\ell^2$-connected graph $G$ has a partition $V(G)=S\cup T$ such that $G[S]$ and $G[T]$ are $\ell$-connected and $d_T(v)\ge \ell$ for every $v\in S$. And they asked whether the quadratic bound can be replaced by a linear bound. In this paper, we answer this question in the affirmative by proving that every $641\ell$-connected graph $G$ has a partition $V(G)=S\cup T$ such that $G[S]$ and $G[T]$ are $\ell$-connected and $d_T(v)\ge \ell$ for every $v\in S$. The proof integrates the Moser–Tardos resampling algorithm.
\end{abstract}

\noindent \textbf{Keyword:} graph partition, vertex-connectivity,  Moser--Tardos algorithm

\noindent \textbf{AMS subject classifications.} 05C40, 05C70, 05D40

\section{Introduction}
Partition problems with degree conditions have been studied extensively, which consider partitioning the vertex set of a given graph into a prescribed number of parts while inheriting certain degree conditions within or between these parts. One of the foundational such results is a classical theorem by Lov\'asz \cite{Lovasz1966} from 1966, stating that for all integers $s, t \in \mathbb{N}$, every graph $G$ with maximum degree $\Delta(G) \le s+t+1$ admits a partition $V(G)=A\cup B$ such that $\Delta(G[A]) \le s$ and $\Delta(G[B]) \le t$, where the maximum degree is the largest number of edges incident to any vertex.

In contrast, define the minimum degree of $G$ is the smallest number of edges incident to any vertex in $G$, write $\delta(G)$, and define $g(s,t)$ as the smallest integer such that every graph $G$ with $\delta(G)\ge g(s,t)$ admits a vertex partition $V(G)=A\cup B$ satisfying $\delta(G[A])\ge s$ and $\delta(G[B])\ge t$. Thomassen \cite{Thomassen1983} showed in 1983 that $g(s,t)\in\mathbb{N}$ exists for all integers $s,t\ge 1$. In 1996, Stiebitz determined the exact threshold, proving that $g(s,t)=s+t+1$; this bound is tight, as witnessed by the complete graph $K_{s+t+1}$.
 Kühn and Osthus  \cite{KuhnOsthus2003} first proposed to simultaneously bound the quantities $\delta(G[A])$, $\delta(G[B])$ and $d_B(v)$ for all $v\in A$. They proved that for any positive integer $\ell$,  every graph $G$ with $\delta(G)\ge {h(\ell)}${, where $h(\ell)=O(\ell^2)$,} admits a partition $V(G)=A{\cup}B$ satisfying $\delta(G[A])\ge \ell$, $\delta(G[B])\ge \ell$, and $d_B(v)\ge \ell$ for all $v\in A$. Furthermore, Ma and Wu improved the quadratic bound $h(\ell)$ to a linear bound in 2024.

\begin{theorem}\cite{MaWu2024}
There exists an absolute constant $C>0$ such that {for every positive integer $k$ and} for every graph $G$ with $\delta(G)\ge Ck$, there exists a 2-partition $V(G)=S{\cup}T$ satisfying $
\delta(G[S])\ge k$, $ \delta(G[T])\ge k$, and $d_T(v)\ge k$ for all $v\in S$.
\end{theorem}

Partitioning under connectivity constitutes an important extension of graph partition theory. Győri and Lovász first established fundamental theorems for connected partitions \cite{Gyori1978, Lovasz1977}: if $G$ is a $k$-connected graph, given distinct vertices $v_1,v_2,\dots,v_k$ and a positive integer sequence $n_1,n_2,\dots,n_k$ with $\sum_{i=1}^k n_i=|V(G)|$, there exists a vertex partition $V(G)=V_1{\cup}\dots{\cup}V_k$ such that $v_i\in V_i$, $|V_i|=n_i$, and each induced subgraph $G[V_i]$ is connected. Define $f(s,t)$ {to} be the minimum integer such that every $f(s,t)$-connected graph admits a 2-partition whose two parts induce $s$-connected and $t$-connected subgraphs, respectively. Furthermore, in 1983, Thomassen \cite{Thomassen1983} proved that for all positive integers $s,t$, there exists a finite integer $f(s,t)$. He also proposed the classical connected partition conjecture.

\begin{conjecture}\label{Thomassen1983}
$f(s,t)=s+t+1$.
\end{conjecture}

This conjecture remains open in the general case. However, Hajnal derived the explicit linear upper bound $f(s,t)\le 4s+4t-13$ for $s,t\ge 3$ \cite{Hajnal1983}. In 2022, Liu and Xu \cite{LiuXu2022} {improved the asymptotic coefficient of this bound by proving} $f(s,t)\le \left\lceil\frac{19(s-1)}{6}\right\rceil+\left\lceil\frac{19(t-1)}{6}\right\rceil+1,
$ when $\min\{s,t\}\ge 3$.  K\"uhn and Osthus~\cite{KuhnOsthus2003} considered the following strengthening of the connectivity problem.
\begin{theorem}\cite{KuhnOsthus2003}
There exists $k(\ell)=O(\ell^2)$ such that every $k(\ell)$-connected graph $G$ admits a partition $V(G)=A{\cup}B$ satisfying $\kappa(G[A])\ge \ell$, $\kappa(G[B])\ge \ell$, and $d_B(v)\ge \ell$ for all $v\in A$.
\end{theorem}

Indeed, K\"uhn and Osthus remarked that it would be interesting to know whether the quadratic bounds on $k$ in the above theorems can be replaced by linear ones.
\begin{problem}\cite{KuhnOsthus2003}
Does there exist $k(\ell)=O(\ell)$ such that every $k(\ell)$-connected graph $G$ admits a partition $V(G)=A{\cup}B$ satisfying $\kappa(G[A])\ge \ell$, $\kappa(G[B])\ge \ell$, and $d_B(v)\ge \ell$ for all $v\in A$?
\end{problem}

In this paper, we answer this question in the affirmative by proving the following main theorem.

\begin{theorem}\label{thm:main}
For every positive integer $\ell$, every $641\ell$-connected graph $G$ has a 2-partition $V(G)=S\cup T$ such that $G[S]$ and $G[T]$ are $\ell$-connected and $d_T(v)\ge \ell$ for every $v\in S ${.}
\end{theorem}

The order of magnitude in Theorem~\ref{thm:main} is best possible. Indeed, $K_{2\ell+1}$ is $2\ell$-connected, while its vertex set cannot be partitioned into two sets inducing $\ell$-connected graphs.

By applying Theorem \ref{thm:main} iteratively \(t-1\) times to \(G[S]\), we obtain the following corollary. The proof is therefore omitted.

\begin{corollary}\label{cor1}
For any two positive integers $t,\ell$, every $641^{t-1}\ell$-connected graph $G$ has a $t$-partition $V(G)=V_1\cup V_2\cup \cdots \cup V_t$ such that  $G[V_i]$ is $\ell$-connected for each $i\in [t]$ and $d_{V_t}(v)\ge \ell$ for every $v\in V(G)\setminus V_t $.
\end{corollary}

Theorem \ref{thm:main} can be applied to establish the existence of non-separating structures in highly connected graphs. The property that the non-separating requirement preserves a c-connected subgraph after removing a prescribed subdivision is also a classical topic. A result in this direction is a theorem of Thomassen \cite{Thomassen1981}, which asserts that every $(c+3)$-connected graph $G$ contains an induced cycle $C$ for which $G-V(C)$ is $c$-connected. Combining Theorem \ref{thm:main} with a result of Mader, which states that every graph of average degree at least $d_H$ contains a subdivision of a prescribed graph $H$ (see e.g. \cite{Diestel1997}, Theorem 3.6.1), we immediately obtain Corollary \ref{cor2} below, where the constant \(k=k(c,H)=641\max\{c,d_H\}\) in Corollary \ref{cor2} is better than the corresponding constant \(k=k(c,H)=2^{16}\max\{c,d_H\}^2\) in Corollary 3 of {K\"uhn} and Osthus \cite{KuhnOsthus2003}.

\begin{corollary}\label{cor2}
For every $c\in\mathbb{N}$ and every graph $H$, there exists $k=k(c,H)\in\mathbb{N}$ such that every $k$-connected graph $G$ contains a subdivision $T_H$ of $H$ for which $G-V(T_H)$ is $c$-connected.
\end{corollary}

Similarly, combining Theorem \ref{thm:main} with a result from \cite{KuhnOsthus} stating that for every $s\in\mathbb{N}$ and every graph $H$ such that every $K_{s,s}$-free graph $G$ with average degree at least $d(s,H)$ contains an induced subdivision of $H$, we obtain Corollary \ref{cor3}, where the constant \(k=k(c,s,H)=641\max\{c,d(s,H)\}\) in Corollary \ref{cor3} is better than the corresponding constant $ k=k(c,s,H)=2^{16}\max\{c,d(s,H)\}^2$ in Corollary 4 of {K\"uhn} and Osthus \cite{KuhnOsthus2003}.

\begin{corollary}\label{cor3}
For all $c,s\in\mathbb{N}$ and every graph $H$, there exists $k=k(c,s,H)\in\mathbb{N}$ such that every $k$-connected $K_{s,s}$-free graph $G$ contains an induced subdivision $T_H$ of $H$ for which $G-V(T_H)$ is $c$-connected.
\end{corollary}

The rest of the paper is {organised} as follows. Section 2 introduces the notation and preliminary lemmas used throughout the paper. Section 3 presents the proof of Theorem \ref{thm:main}. Finally, we conclude the paper with some remarks in Section 4.

\section{Preliminaries}
All graphs are finite and simple. For a graph $G$, we write $V(G)$ and $E(G)$ for its \textbf{vertex set} and \textbf{edge set}, respectively, and $e(G)=|E(G)|$. For $X\subseteq V(G)$, let $G[X]$ denote the subgraph induced by $X$. For $v\in V(G)$ and $X\subseteq V(G)$, put $d_X(v)=|N_G(v)\cap X|$. The \textbf{minimum degree} of $G$ is defined as $\delta(G)=\min\{d_G(v)\mid v\in V(G)\}$, and the \textbf{maximum degree} as $\Delta(G)=\max\{d_G(v)\mid v\in V(G)\}$. A graph is \textbf{$k$-connected} if it has at least $k+1$ vertices and remains connected after the deletion of fewer than $k$ vertices. We write $\kappa(G)$ for the vertex-connectivity of $G$. {We use $\mathbb N=\{1,2,\ldots\}$ and $[t]=\{1,\ldots,t\}$ for every $t\in\mathbb N$.}

Given a graph $G$ and two vertices $x,y\in V(G)$, we denote the edge joining $x$ and $y$ as $xy$. For a vertex $x$ of $G$, we define $N_G(x)=\{y\mid xy\in E(G)\}$ as the \textbf{{neighbourhood}} of $x$, and $d_G(x)=|N_G(x)|$ as the \textbf{degree} of $x$.
Let $d(G)=\frac{2e(G)}{|V(G)|}$ denote the \textbf{average degree} of $G$
and let $\mad(G)=\max_{\varnothing\ne J\subseteq G} d(J)$.  A graph $G$ is $d$-degenerate if every nonempty subgraph of $G$ contains a vertex of degree at most $d$. If $\mad(G)\le d$, where $d$ is a nonnegative integer, then we can deduce that $G$ is {$d$-degenerate}. Thus $V(G)$ has an ordering $v_1,\ldots,v_n$ such that every $v_i$ has at most $d$ neighbours in $\{v_{i+1},\ldots,v_n\}$.

We shall use the following lemma of K\"uhn and Osthus~\cite{KuhnOsthus2003}.

\begin{lemma}\label{lem:KO}\label{KuhnOsthus2003}
Let $a$ be a positive integer. Every $9a$-connected graph $G$ has a partition $ V(G)=S\cup T $ such that
\begin{enumerate}[label=\textup{(\roman*)}]
\item $G[T]$ is $a$-connected;
\item $d(G[S])\ge a$;
\item $d(J)<32a$ for every subgraph $\emptyset \neq J\subseteq G[S]$.
\end{enumerate}
\end{lemma}

We also use the following theorem of Mader~\cite{Mader1972}.

\begin{theorem}\label{thm:Mader}\cite{Mader1972}
Let $k$ be a positive integer. Every graph of average degree at least $4k$ contains a $k$-connected subgraph.
\end{theorem}

We shall use the standard Chernoff bound.
\begin{lemma}\label{chernoff}
Suppose $a>0$ and $X_1,\dots,X_m$ are independent identically distributed random variables with $\mathbb{P}(X_i=1)=p$ and $\mathbb{P}(X_i=0)=1-p$. Let $\mu = pm$. Then $\mathbb{P}\left(\sum_{i=1}^m X_i < \mu - a\right) < e^{-a^2/(2\mu)}${.}
\end{lemma}

One useful {identity} is stated below.

\begin{theorem}[Tonelli, see \cite{folland1999real}]
If \(Y_0,Y_1,\ldots\) are nonnegative random variables, then
\[
\mathbb E\Big[\sum_{t\ge0}Y_t\Big]
=
\sum_{t\ge0}\mathbb EY_t\text{{.}}
\]
\end{theorem}

We apply the Moser--Tardos theorem in place of the classical Lov\'asz Local Lemma, since the resampling process provides direct control of the marginal probabilities of the final assignment.

\begin{theorem}[Moser--Tardos \cite{MoserTardos2010}]\label{moser}
Let $X_1,\dots,X_n$ be independent random variables.
Let $\mathcal A=\{A_1,\dots,A_m\}$ be a collection of bad events.
For each $A_i$, let  $D_i$ be the set of events $A_j\neq A_i$ sharing at least one variable with $A_i$.
If there exist real numbers $0<x_i<1$ such that
\begin{equation}\label{eq1}
  \mathbb P(A_i)\le x_i\prod_{A_j\in D_i}(1-x_j)\quad \text{for all }i,
\end{equation}
then the Moser--Tardos resampling algorithm terminates in expected at most $\frac{x_i}{1-x_i}$ resampling steps for every $A_i$ and produces an assignment avoiding all events $A_i$.
\end{theorem}

The following elementary observation is also needed.

\begin{lemma}\label{lem:addvertex}
Let $H$ be an $r$-connected graph and let $v\notin V(H)$. If $v$ has at least $r$ neighbours in $H$, then the graph obtained from $H$ by adding $v$ and its incident edges to $H$ is $r$-connected.
\end{lemma}

\begin{proof}
Let $Z$ be a set of fewer than $r$ vertices. If $v\in Z$, then $H-(Z\setminus\{v\})$ is connected. Suppose that $v\notin Z$. Then $H-Z$ is connected, and at least one neighbour of $v$ in $H$ does not belong to $Z$. Hence the graph obtained after deleting $Z$ is connected.
\end{proof}

\section{Proof of Theorem \ref{thm:main}}

This section contains the proof of Theorem \ref{thm:main}. A brief outline of the proof is presented first. Suppose $G$ is a $641\ell$-connected graph. By Lemma \ref{lem:KO}, there exists an initial partition $V(G)=S_0\cup T_0$ such that $G[T_0]$ is highly connected, $G[S_0]$ has large average degree, and every subgraph of $G[S_0]$ has bounded average degree. Hence $G[S_0]$ admits a degeneracy ordering. Every vertex $v\in S_0$ with fewer than $\ell$ {neighbours} in $T_0$ has many {neighbours} preceding $v$ in this ordering. We select vertices of $S_0$ independently with probability $1/4$. By the Moser--Tardos algorithm, there exists a vertex selection such that every vertex $v\in S_0$ with $d_{T_0}(v)<\ell$ has at least $\ell$ selected preceding {neighbours}. Moreover, the Moser--Tardos algorithm yields a subset $M\subseteq S_0$ satisfying $d(G[S_0\setminus M])>4\ell$. We then add the vertices in $M$ to $T_0$ in degeneracy order to verify that $G[T_0\cup M]$ is $\ell$-connected. Theorem \ref{thm:Mader} guarantees an $\ell$-connected subgraph ${H}$ of $G[S_0\setminus M]$. {Let $S=V(H)$.} Finally, we add the vertices in $S_0\setminus (M\cup S)$ to $T_0\cup M$ to obtain the desired set $T$, which completes the proof.

To {prove} Theorem \ref{thm:main}, we establish the following vertex redistribution lemma.

\begin{lemma}\label{lem:transfer}
Let $r$ be a positive integer, and let $G$ be a graph with $\delta(G)\ge641r${.}
Suppose that $V(G)=S\cup T$ satisfying $G[T]$ is $r$-connected,  $d(G[S])\ge16r$, and  $\mad(G[S])\le512r$. Then there is a partition $V(G)=S'\cup T'$
with $S'\subseteq S$ and $T\subseteq T'$ such that
\begin{enumerate}[label=\textup{(\roman*)}]
\item $G[T']$ is $r$-connected;
\item $d_{T'}(v)\ge r$ for every $v\in S'$;
\item $ d(G[S'])>4r$.
\end{enumerate}
\end{lemma}
\begin{proof}By~$\mad(G[S])\le512r$, the graph $G[S]$ is $512r$-degenerate. Fix an ordering $S=\{v_1,\ldots,v_m\}$ such that every $v_i$ has at most $512r$ neighbours in $\{v_{i+1},\ldots,v_m\}$. Call $v_i$ bad if $d_T(v_i)<r${.}
For every bad vertex \(v_i\), we have \(d_S(v_i)\ge 641r-(r-1)=640r+1\). By the degeneracy ordering, \(v_i\) has at most \(512r\) {neighbours} appearing after it. It follows that \(v_i\) has at least \(128r+1\) {neighbours} in \(\{v_1,\ldots,v_{i-1}\}\). For each bad vertex \(v_i\), fix a subset \(P_i\subseteq N_G(v_i)\cap\{v_1,\ldots,v_{i-1}\}\) with \(|P_i|=128r\):
\begin{equation}\label{eq:Pi}
P_i\subseteq N_G(v_i)\cap\{v_1,\ldots,v_{i-1}\},\qquad |P_i|=128r\text{{.}}
\end{equation}
Independently select each vertex of \(S\) with probability \(1/4\) to form a candidate set \(M_0\). For every bad vertex \(v_i\), define the bad event $\mathcal E_i=\big\{\,|P_i\cap M_0|<r\,\big\}${.}
Observe that \(|P_i\cap M_0|\) follows a binomial distribution with parameters \((128r,1/4)\), so its expectation equals \(32r\). By the Chernoff bound (Lemma \ref{chernoff}),
\begin{equation}\label{eq:qi}
\Prob(\mathcal E_i) \le \exp\left(-\frac12\left(\frac{31}{32}\right)^2 32r\right)
= \exp\left(-\frac{961r}{64}\right)=:q_r\text{{.}}
\end{equation}
 We next consider the dependency graph of the events \(\mathcal E_i\). A variable corresponding to \(v_j\) can appear in \(\mathcal E_i\) only if \(v_j\in P_i\). In this case \(j<i\) and \(v_iv_j\in E(G)\), so \(v_i\) is a later {neighbour} of \(v_j\). Thus every variable is contained in at most \(512r\) events. Since each event depends on \(128r\) variables, every event is adjacent to at most \(D:=65536r^2\) other events. Set \(x=2q_r\). The function \(r^2e^{-961r/64}\) is decreasing for \(r\ge1\),  combining this with \(D:=65536r^2\), we obtain
\[
Dx\le 2\cdot 65536\,e^{-961/64}<0.04\text{.}
\]
By Bernoulli's inequality, \((1-t)^k \ge 1-kt\) for all \(t\in[0,1]\) and integers \(k\ge 0\). Setting \(t:=x\) and \(k:=D\), we obtain \((1-x)^D \ge 1-Dx>0.96\), and hence $\Prob(\mathcal E_i)\le q_r<x(1-x)^D$.
Thus~\eqref{eq1} holds with $x_i=x$ for every $i$.

Fix any total order $\sigma$ on the events $\mathcal E_i$ and run the Moser--Tardos algorithm using the order $\sigma$. At each step, among all currently violated events, we select the minimal event with respect to \(\sigma\) and resample its dependent variables. For each $v\in S$, let $\xi_{v,0},\xi_{v,1},\xi_{v,2},\ldots$
be independent Bernoulli$(1/4)$ random variables assigned a value $0$ or $1$, and assume that all variables in  $\{\xi_{v,0},\xi_{v,1},\xi_{v,2},\ldots \mid v\in S\}$ are mutually independent. Initially $X_v=\xi_{v,0}$. Whenever $X_v$ is resampled for the $j$-th time, its new value is $\xi_{v,j}$. For each integer $t\ge 0$, let $\mathcal{F}_t$ denote the $\sigma$-algebra generated by the entries of the resampling table exposed prior to the $(t+1)$-th resampling round and the entire resampling history accumulated up to this step. For each event $\mathcal E_i$, let $\mathrm{vbl}(\mathcal E_i)$ denote the set of variables on which $\mathcal E_i$ depends. For $v\in S$, let $R_v(t)$ be the number of times $X_v$ has been resampled before the $(t+1)$-st step. If the $(t+1)$-st resampling step occurs, let $J_t$ denote the index of the bad event resampled at this step, and define
\[
I_t(v)=\begin{cases}
         \mathbf 1, & \mbox{if } \text{the $(t+1)$-st resampling step occurs and }v\in\operatorname{vbl}(\mathcal E_{J_t}) \\
         \mathbf 0, & \mbox{otherwise{.}}
       \end{cases}
\]
If $I_t(v)=1$, the value assigned to $X_v$ in this resampling step is $\xi_{v,R_v(t)+1}$. This random variable $\xi_{v,R_v(t)+1}$ has not been {realised} previously and is independent of $\mathcal F_t$. Hence
\begin{equation}\label{eq:conditionaltable}
\Ex\left[I_t(v)\xi_{v,R_v(t)+1}\mid\mathcal F_t\right]
=\frac14\,I_t(v)\text{{.}}
\end{equation}
By Theorem \ref{moser}, the resampling process terminates with probability 1; the expected number of times each bad event is resampled is at most $\frac{x}{1-x}<3q_r$. Let $M$ denote the set of  vertices whose final output value from the algorithm is 1.  Since no bad event occurs at termination,
\begin{equation}\label{eq:support}
|P_i\cap M|\ge r
\qquad\text{for every bad }v_i\text{{.}}
\end{equation}

For $v\in S$, let $R_v$ be the number of times the corresponding variable is resampled, and let $X_v^*$ be its value at termination. Since this variable occurs in at most $512r$ events,
\begin{equation}\label{eq:Rv}
\Ex R_v
\le512r\cdot3q_r
=1536rq_r
=:\eta_r\text{{.}}
\end{equation}

\begin{claim}\label{eq:marginalapp}
$\Prob(v\in M)=\Prob(X_v^*=1)\le \frac14(1+\eta_r)$, for every $v\in S$.
\end{claim}
\begin{proof}
The quantity $X_v^*$ denotes the value obtained from the final sampling of $X_v$. If $X_v^*=1$, then either $\xi_{v,0}=1$, or one of $\xi_{v,0},\xi_{v,1},\xi_{v,2},\ldots$ actually used in a resampling of $X_v$ is equal to $1$; hence, pointwise,
\[
X_v^*\le \xi_{v,0}+\sum_{t\ge0}I_t(v)\xi_{v,R_v(t)+1}\text{{.}}
\]
 By Tonelli's theorem and~\eqref{eq:conditionaltable},
\begin{align*}
\Prob(X_v^*=1)
&\le \Ex\left[\xi_{v,0}\right]+\sum_{t\ge0}\Ex\bigl[I_t(v)\xi_{v,R_v(t)+1}\bigr]
=\frac14+\frac14\sum_{t\ge0}\Ex I_t(v)
=\frac14\bigl(1+\Ex R_v\bigr)\text{{.}}
\end{align*}
This establishes the claim.
\end{proof}

For every edge $uv\in E(G[S])$,~Claim \ref{eq:marginalapp} gives $\Prob(u,v\notin M)\ge1-\Prob(u\in M)-\Prob(v\in M)\ge\frac{1-\eta_r}{2}$.
Therefore
\begin{equation}\label{eq:edgesremain}
\Ex e(G[S\setminus M])
\ge\frac{1-\eta_r}{2}e(G[S])\text{{.}}
\end{equation}
The fact that $d(G[S])\ge16r$ implies $e(G[S])\ge8r|S|$.
Combining this with~\eqref{eq:edgesremain}, we obtain $\Ex e(G[S\setminus M])\ge4(1-\eta_r)r|S|$.
\begin{claim}
There exists {such a set} $M$ such that $d(G[S\setminus M])>4r$.
\end{claim}
\begin{proof}
  Define $\Phi(M)=e(G[S\setminus M])-2r|S\setminus M|$. The function $re^{-961r/64}$ is decreasing for $r\ge1$, which implies $\eta_r\le1536e^{-961/64}<\frac1{2000}$ and so $2-4\eta_r>0$.
Since $|S\setminus M|\le|S|$, we have $\Ex\Phi(M)\ge(2-4\eta_r)r|S|>0$.
This means that there is an $M$ for which $\Phi(M)>0$, which implies $e(G[S\setminus M])>2r|S\setminus M|$. {Since $\Phi(M)>0$, the set $S\setminus M$ is nonempty, and hence $d(G[S\setminus M])>4r$.}
\end{proof}

Put $S'=S\setminus M$, $T'=T\cup M$. Therefore, $ d(G[S'])>4r$.
This proves~\textup{(iii)}. To show~\textup{(i)}, write $M=\{v_{i_1},\ldots,v_{i_s}\}$, where  $i_1<\cdots<i_s$.
Starting with $T$, add $v_{i_1},\ldots,v_{i_s}$ in this order. If $v_{i_j}$ is not bad, then $d_T(v_{i_j})\ge r$. If $v_{i_j}$ is bad, then~\eqref{eq:support} and~\eqref{eq:Pi} imply that $v_{i_j}$ has at least $r$ neighbours in $M$ which precede $v_{i_j}$. These vertices have already been added. Thus every vertex has at least $r$ neighbours in the set constructed before it is added. Repeated application of Lemma~\ref{lem:addvertex} shows that $G[T']$ is $r$-connected.
It remains to prove~\textup{(ii)}. Let $v_i\in S'$. If $v_i$ is not bad, then $d_{T'}(v_i)\ge d_T(v_i)\ge r$.
If $v_i$ is bad, then~\eqref{eq:support} gives $d_{T'}(v_i)\ge|P_i\cap M|\ge r$.
This completes the proof.
\end{proof}

\begin{proof} \textbf{of Theorem~\ref{thm:main}}
Let $G$ be a $641\ell$-connected graph, and set $a=16\ell${.}
Since $641\ell\ge9a$, Lemma~\ref{lem:KO} gives a partition $V(G)=S_0\cup T_0$
such that $G[T_0]$ is $16\ell$-connected, $ d(G[S_0])\ge16\ell,$
and $\mad(G[S_0])<512\ell$. Also $\delta(G)\ge\kappa(G)\ge641\ell${.}
We may therefore apply Lemma~\ref{lem:transfer} with $r=\ell$. We obtain a partition $V(G)=S_1\cup T_1$
such that $G[T_1]$ is $\ell$-connected,
\begin{equation}\label{eq:crossfinal}
d_{T_1}(v)\ge\ell
\qquad\text{for every }v\in S_1,
\end{equation}
and $d(G[S_1])>4\ell${.}
By Theorem~\ref{thm:Mader}, $G[S_1]$ contains an $\ell$-connected subgraph $H$. Let $S=V(H)$ and $T=V(G)\setminus S ${.}
Since $H$ is a spanning subgraph of $G[S]$, the graph $G[S]$ is $\ell$-connected. Now $T=T_1\cup(S_1\setminus S)${.} By~\eqref{eq:crossfinal}, every vertex of $S_1\setminus S$ has at least $\ell$ neighbours in $T_1$. Starting with $G[T_1]$ and adding the vertices of $S_1\setminus S$ one at a time, Lemma~\ref{lem:addvertex} shows that $G[T]$ is $\ell$-connected. Finally, for every $v\in S$, $d_T(v)\ge d_{T_1}(v)\ge\ell$
by~\eqref{eq:crossfinal}. This proves the theorem.
\end{proof}

\section{Remarks}

The constant $641$ in the present paper has not been {optimised} and is chosen merely for computational simplicity. It would be interesting to determine the best possible constant in Theorem \ref{thm:main}.

Furthermore, the following {problem} on connectivity partitions in digraphs remains open.
\begin{problem}\cite{KuhnOsthusTownsend2016}
Is there a function $g$ such that, for every positive integer $k$, every strongly $g(k)$-connected digraph can be partitioned into two sets inducing strongly $k$-connected digraphs?
\end{problem}

For tournaments, the corresponding problem has recently been settled with an asymptotically optimal linear bound by Girão and Letzter \cite{girao}.

\textbf{Declaration about the use of AI:} ChatGPT was used for the probabilistic calculations in the proof of Lemma 3.1 in Section 3 and to improve the grammar of several paragraphs of the text. After using this service, the authors reviewed and edited the content as needed and take full responsibility for the content of the manuscript.

\end{document}